\documentclass[12pt]{amsart}
\usepackage{graphics}
\usepackage{latexsym}
\usepackage[dvips]{epsfig}
\usepackage{color}
\usepackage{graphicx}
\usepackage{amsmath}
\usepackage{amssymb}
\usepackage{amsfonts}
\usepackage{eufrak}
\usepackage{amstext}
\usepackage{amsopn}
\usepackage{amsbsy}
\usepackage{amscd}
\usepackage{amsxtra}
\usepackage{amsthm}
\usepackage{bm}
\usepackage{CJK}

\newcommand{\R}{{\mathbb R}}

\newcommand{\beq}{\begin{equation}}
\newcommand{\eeq}{\end{equation}}
\newcommand{\ben}{\begin{eqnarray}}
\newcommand{\een}{\end{eqnarray}}
\newcommand{\beno}{\begin{eqnarray*}}
\newcommand{\eeno}{\end{eqnarray*}}

\newtheorem{thm}{Theorem}[section]

\newtheorem{lem}[thm]{Lemma}

\newtheorem{rmk}[thm]{Remark}

\begin{document}

\title[solutions with polynomial growth]{On solutions with polynomial growth to an autonomous nonlinear elliptic problem}

\author{ Kelei Wang}
 \address{Wuhan Institute of Physics and Mathematics,
The Chinese Academy of Sciences, Wuhan 430071, China. Email: wangkelei@wipm.ac.cn}

\author{Juncheng Wei}
\address{Department of Mathematics, The Chinese University of Hong Kong, Shatin, Hong Kong. Email: wei@math.cuhk.edu.hk}

\begin{abstract}
We study the following nonlinear elliptic problem
\[ -\Delta u =F^{'} (u)  \ \mbox{in} \ \R^n\]
where $F(u)$ is a periodic function. Moser (1986) showed that for any minimal and nonself-intersecting solution, there exist $ \alpha \in \R^n$ and $ C>0$ such that
\[(*) \ \ \ \  \ \  \ \ \ \ \ | u- \alpha \cdot x | \leq C.\]
He also showed the existence of solutions with any prescribed $\alpha \in \R^n$. In this note, we first prove that any solution satisfying (*) with nonzero vector $\alpha$ must be one dimensional. Then we show that in $\R^2$, for any positive integer $d\geq 1$ there exists a solution  with polynomial growth $|x|^d$.

\end{abstract}



\maketitle

\date{}

\section{Introduction and Main Results}
\setcounter{equation}{0}

In search of analogue of  Aubry-Mather theory for quasilinear partial differential equations in $\R^n$, Moser \cite{m}  studied the following equation
\begin{equation}
\label{0}
\sum_{i=1}^n \frac{\partial}{\partial x_i} {\mathcal F}_{p_i} (x, u, Du)- {\mathcal F}_u (x, u, Du)=0
\end{equation}
which is the Euler-Lagrangian equation for the functional
\begin{equation}
\int_{\R^n} {\mathcal F} (x, u, Du) dx
\end{equation}
where $ {\mathcal F}$ is $ 1-$periodic in all variables $ x_1, ..., x_n$ and $u$, elliptic and of quadratic growth in $p= Du$.

 A solution  $ u(x)$  of (\ref{0}) is called {\em minimal} if
\begin{equation}
\int_{\R^n} \Big[ {\mathcal F} (x, u+\varphi, Du+D \varphi) dx - {\mathcal F} (x, u, Du) \Big] dx \geq 0, \ \forall \ \varphi \in C_0^\infty (\R^n).
\end{equation}
A solution of (\ref{0}) is said  to be {\em without self intersections} or WSI if (i) for each $j \in {\mathbb Z}^n$ and $ j_{n+1} \in {\mathbb Z}, u(x+j)- u(x) - j_{n+1}$ does not change sign for $ x \in \R^n$, or (ii) for some $ j \in {\mathbb Z}^n$ and $ j_{n+1} \in {\mathbb Z}, u(x+j)\equiv u(x) + j_{n+1}$.

For  minimal and WSI   solutions to (\ref{0}), Moser \cite{m} showed: (1) There exists a unique vector $\alpha \in \R^n$, the so-called {\em rotation vector} and a constant $C$, such that
\begin{equation}
\label{1m}
| u(x) - \alpha \cdot x | \leq C, \ \forall \ x \in \R^n.
\end{equation}
\noindent
(2) Conversely, for every vector $\alpha \in \R^n$ there exists a minimal solution $u$ with rotation vector $\alpha$ and a constant $C$ and satisfying (\ref{1m}).

Moser's paper \cite{m} has received lots of attention in the literature.  Among many results, we mention that  Bangert \cite{ba} showed the existence of  heteroclinic states under some gap conditions, and  Rabinowitz and Stredulinsky \cite{rs1, rs2} developed variational gluing methods for mixed states of Allen-Cahn type equations. (See also \cite{rs3} for non-autonomous case.)  There is also a strong connection between Moser's problem and De Giorgi's conjecture. See Farina and Valdinoci \cite{fv}.  For the latest developments, we refer to the survey paper by Rabinowitz \cite{r} and the references therein.

\medskip

In this note, we consider the autonomous Moser's problem, namely we study the following problem
\begin{equation}
\label{equation}
-\Delta u = F^{'}(u) \ \mbox{in} \ \R^n
\end{equation}
where $ F(u)$ is a  smooth periodic  function. A typical example is the so-called sine-Gordon nonlinearity $ F(u)=1-\cos (u)$.

Our first result is a classification theorem on solutions to (\ref{equation}) satisfying (\ref{1m}).

\begin{thm}
\label{thm1}
Let $u\in C^2(\mathbb{R}^n)$ be a solution of \eqref{equation}. Assume that there exist a nonzero vector $\alpha \in\mathbb{R}^n$
and a constant $C>0$ such that
\begin{equation}\label{condition}
|u(x)-\alpha \cdot x|\leq C~~\text{for}~~\forall x\in\mathbb{R}^n.
\end{equation}
Then  there is a function $v \in C^2 (\R^n)$ such that $u(x)= v(\alpha \cdot x)$.
\end{thm}

In the above theorem, $\alpha \not =0 $ is necessary. In fact for Allen-Cahn or Sine-Gordon equations, there are bounded solutions with multiple transitions (\cite{A-C-M, D-F-P, dkpw}). Theorem \ref{thm1} also holds when $-\Delta u=f(u)$ where $f$ is periodic. Note that it can be directly shown that one dimensional solutions satisfying \eqref{condition} have no self-intersection.

Theorem \ref{thm1} has been proved by Farina and Valdinoci \cite{fv} under the minimality condition. Here we have removed the minimality assumption. Theorem \ref{thm1} shows that unbounded solutions to (\ref{equation}) with linear growth are all one dimensional. Notice that $\alpha \cdot x$ is the simplest nonconstant harmonic function in $\R^n$.  Based on this, J. Byeon and P. Rabinowitz \footnote{Private discussion} asked

\medskip

\noindent
{\em Question:  given any harmonic function, w, on $\R^n$, is there  a solution, u, of (\ref{0}) with $ || u - w ||_{L^{\infty}(\R^n)}$ bounded?}

\medskip

The following theorem answers the question partially.

\begin{thm}
\label{thm2}
Let $n=2$ and $ d \geq 2$. Assume that $F(u)$ is even.  Let $ \varphi (x,y)$ be the real part of the harmonic polynomial $z^d$.  (Here $z=x+i y.$)  Then there exists a solution to (\ref{equation}), enjoying the same symmetry as $\varphi (x,y)$ and satisfying
\begin{equation}
\label{est}
| u(x, y)- \varphi (x, y) | \leq C (1+|z|)^{\frac{3}{2}}.
\end{equation}
Furthermore, for $d \geq 3$  we also have the following improved upper bound:
\begin{equation}
\label{newest}
| u(x, y)- \varphi (x, y) | \leq C (1+|z|)^{2-\frac{d}{2}}.
\end{equation}
\end{thm}

\medskip

\begin{rmk}
If $d \geq 4$, then $\frac{d}{2}\geq 2$. Thus for $d \geq 4$, we answered Byeon-Rabinowitz's question affirmatively, in the autonomous setting  (\ref{equation}).  Note also that for $d >4$, we have better decay estimates. The key to obtain (\ref{newest}) is some oscillatory integral estimate (see (\ref{4.1}) below). For $d=2$ or $3$, this estimate is  not sufficient. We believe that the $L^\infty$ bound should also hold for $d=2,3$.
\end{rmk}

\begin{rmk}
Another interesting question is whether or not the evenness condition is necessary.
\end{rmk}

\medskip

In the rest of the paper, we prove Theorem \ref{thm1} in Section 2, the estimate (\ref{est}) of Theorem \ref{thm2} in Section 3 and the better estimate (\ref{newest}) of Theorem \ref{thm2} in Section 4 respectively.

\medskip

\noindent {\bf Acknowledgment.} The second author thanks Professors J. Byeon and  P. Rabinowitz for suggesting the problem and nice discussions.

\section{Proof of Theorem \ref{thm1}}
 \setcounter{equation}{0}

In this section, we prove Theorem \ref{thm1} by the method of moving planes.

\medskip

Without loss of generality, assume that $|\alpha|=1$ and $\alpha$ is the
$x_n$ direction. We use the notation that $x=(x^\prime,x_n)$ where
$x^\prime\in\mathbb{R}^{n-1}$. For any unit vector $e$ such that
$e\cdot \alpha>0$, we will prove that for every $t\geq 0$,
\begin{equation}\label{sliding goal}
u(x+te)\geq u(x)~~\text{for all}~~x\in\mathbb{R}^n.
\end{equation}
This then implies that $e\cdot\nabla u\geq 0$ in $\mathbb{R}^n$. By continuity, this also holds for $e$ and $-e$, if $e\cdot \alpha=0$, which then implies that $e\cdot\nabla u\equiv 0$ and that $u$ depends only on $\alpha\cdot x$.

For any $t>0$, define $u^t(x)=u(x+te)$. First we note that, since $e_n=e\cdot \alpha>0$, for $t$ large, by \eqref{condition},
\[u^t(x)\geq x_n+te_n-C\geq x_n+C\geq u(x).\]
Hence we can define
\[t_0:=\inf\{t: \forall s\geq t, \eqref{sliding goal}~~\text{holds}\}.\]

\medskip

Assuming that $t_0>0$, we will get a contradiction. First note that $u^{t_0}\geq u$ by continuity.
It is impossible to have $u^{t_0}\equiv u$, because this would imply that $u$ is $t_0$ periodic in the $e$
 direction, which contradicts \eqref{condition}. ($e\cdot \alpha>0$ implies that $u$ goes to infinity when $x$ goes to infinity along the $e$ direction.)
Hence by the strong maximum principle we have
\begin{equation}\label{2.6}
u^{t_0}> u.
\end{equation}

By the definition of $t_0$, there exists $t_k<t_0$ such that
\[\inf_{\mathbb{R}^n}(u^{t_k}-u)<0.\]
In particular, there exists $x_k\in\mathbb{R}^n$ such that
\begin{equation}\label{2.7}
(u^{t_k}-u)(x_k)<0.
\end{equation}

Assume the period of $F(u)$ is $T$. By \eqref{condition}, we can take a constant $a_k$, which is a multiple of $T$ such that
\[u_k(x):=u(x+x_k)-a_k\]
satisfies $|u_k(0)|\leq T$. \eqref{2.6} and \eqref{2.7} imply respectively that
\begin{equation}\label{2.8}
u_k^{t_0}> u_k.
\end{equation}
\begin{equation}\label{2.9}
(u_k^{t_k}-u_k)(0)<0.
\end{equation}
Note that $u_k$ still satisfies \eqref{condition} with a larger constant $2C+T$, which is independent of $k$. By the elliptic regularity, $u_k$ is uniformly bounded in $C^3(B_R(0))$ for any $R>0$. Hence we can take a subsequence of $u_k$
such that $u_k$ converges to $u_\infty$ in $C^2(B_R(0))$ for any $R>0$. Letting $k\to+\infty$ in \eqref{2.8} and \eqref{2.9}, we get
\[u_\infty^{t_0}\geq u_\infty,~~~u_\infty^{t_0}(0)=u_\infty(0).\]
By the strong maximum principle, $u_\infty^{t_0}\equiv u_\infty$. That is, $u_\infty$ is $t_0$ periodic along the direction $e$. Since $u_\infty$ satisfies \eqref{condition}, this is a contradiction and also finishes the proof of Theorem \ref{thm1}.


\section{Proof of Theorem \ref{thm2}}
\setcounter{equation}{0}

In this section, we prove the existence of solutions satisfying estimate (\ref{est}) in Theorem \ref{thm2}.

 We denote $z=x+iy\in\mathbb{C}$. We also identify $z=re^{i\theta}$ with $(x,y)\in\mathbb{R}^2$. Let $d \geq 2$ be a positive integer and
$\varphi (x, y)=Re(z^d)$. Denote $G$ the
rotation of order $2d$. Note that $\varphi(Gz)=-\varphi(z)$.

Let $D=\{-\frac{\pi}{2d}<\theta<\frac{\pi}{2d}\}$ be a nodal domain of $\varphi$.
For every $R>0$, take $D_R=B_R(0)\cap D$ and $u^R$ to be a minimizer of the functional
\[\int_{D_R}\frac{1}{2}|\nabla u|^2+F(u),\]
with the Dirichlet boundary condition $u=\varphi$ on $\partial D_R$.

\medskip

First, the minimizer exists since $F(u)$ is a bounded periodic function. Second, we may assume that $ u^R \geq 0$ in $D_R$ since otherwise we may replace  the minimizer with $ |u^R|$ (noting that $F$ is even and $ F(|u|)=F(u)$). Since $F^\prime(u)=0$, the strong maximum principle implies that $u_R>0$ in $D_R$.
Once again by the oddness of $F^\prime(u)$ and the fact that $F^\prime(0)=0$,  by rotational symmetry of $\frac{2\pi}{d}$, $u^R$ can be extended to $B_R(0)$
and it satisfies the equation $ -\Delta u= F^{'} (u)$  in $B_R(0)$. By construction, $u^R$ has the same symmetry as $\varphi$, that is, $u^R(Gz)=-u^R(z)$ for $z\in B_R(0)$.\footnote{Another method to get $u^R$ is to find a minimizer of $\int_{B_R}\frac{1}{2}|\nabla u|^2+F(u)$ in the invariant class $\{u=\varphi~~\text{on}~~\partial B_R,~~\text{and}~~u(Gz)=-u(z)\}$. $u^R$ can be proved to satisfy $ -\Delta u= F^{'} (u)$ in $B_R(0)$ by the heat flow method. Note that because $F(u)$ is even, the invariant class is positively invariant by the heat flow.} In particular, the nodal domain of $u^R$ is the same with $\varphi$ and $\{u^R=0\}$ is composed by $2d$ rays with the form $re^{i\frac{k\pi}{2d}}$ for $k=1,3,\cdots,4d-1$ and $r\in[0,R]$.

For any $r\in(0,R)$, let $\varphi^r$ be the solution of
\begin{equation*}
\left\{ \begin{aligned}
&\Delta \varphi^r=0,~~\text{in}~~B_r,\\
&\varphi^r=u^R,~~\text{on}~~\partial B_r.
                          \end{aligned} \right.
\end{equation*}
Since $u^R$ has the same symmetry as $\varphi$, by the uniqueness of the solution to the above problem, $\varphi^r$ has the same symmetry as $\varphi$, and $\{\varphi^r=0\}$ is composed by $2d$ rays of the form $re^{i\frac{k\pi}{2d}}$ for $d=1,3,\cdots,4d-1$ and $r\in[0,r]$. This implies that $\varphi^r=u^R$ on $\partial D_r$ and $\varphi^r$ is also the harmonic extension of $u^R$ from $\partial D_r$ to $D^r$.

\medskip

\begin{lem}
There exists a constant $C$, independent of $r$ and $R$, such that
\begin{equation}\label{comparison of energy}
\int_{B_r(0)}|\nabla\varphi^r-\nabla u^R|^2\leq Cr^2.
\end{equation}
\end{lem}

\medskip

Since we expect $u^R$ grows like $|z|^d$ and $|\nabla u^R|$ grows like $|z|^{d-1}$ with $d\geq 2$, this estimate implies that $u^R$ and $\varphi^r$ are close to each other (after a rescaling) at large scale. Below we will use this inequality to estimate the error $u^R-\varphi^r$.
\begin{proof}
By the minimality of $u^R$, we have
\[\int_{D_r}\frac{1}{2}|\nabla u^R|^2+F(u^R)\leq\int_{D_r}\frac{1}{2}|\nabla \varphi^r|^2+F(\varphi^r)\]
which implies
\begin{equation}
\label{dr}
 \int_{D_r}|\nabla u^R|^2- |\nabla \varphi^r |^2 \leq C r^2
\end{equation}
since $F$ is a bounded periodic function.

\medskip

On the other hand, an integration by parts using the fact that $u^R=\varphi^r$ on $\partial D_r$ shows that
\[\int_{D_r}|\nabla\varphi^r-\nabla u^R|^2=\int_{D_r}|\nabla u^R|^2-|\nabla\varphi^r|^2.\]
 Substituting the above equality into the inequality (\ref{dr}), we get \eqref{comparison of energy}.
\end{proof}
\begin{lem}\label{lem 2}
There exists a constant $C$, independent of $r$ and $R$, such that for all $0<r<R$,
\[\sup_{B_{r/2}(0)}|\varphi^r-u^R|\leq Cr^{3/2}.\]
\end{lem}
\begin{proof}
We will assume that $r$ is large enough.
Let $\bar{u}^r(z):=\frac{1}{r^d}u^R(rz)$ and $\bar{\varphi}^r(z):=\frac{1}{r^d}\varphi^r(rz)$ for $z\in B_1(0)$. By \eqref{comparison of energy},
\[\int_{B_1(0)}|\nabla\bar{\varphi}^r-\nabla \bar{u}^r|^2\leq Cr^{2-2d}.\]
Since $\bar{u}^r=\bar{\varphi}^r$ on $\partial B_1(0)$, by the
Poincare inequality,
\begin{equation}\label{1}
\int_{B_1(0)}|\bar{\varphi}^r-\bar{u}^r|^2\leq Cr^{2-2d}.
\end{equation}
 Note that
\begin{equation}\label{rescaled equation}
|\Delta(\bar{\varphi}^r-\bar{u}^r)|=|r^{2-d}F^{\prime}(r^d\bar{u}^r)|\leq Cr^{2-d}.
\end{equation}
 Take
a $r_0\in(3/4,1)$ such that
\[\int_{\partial B_{r_0}(0)}|\bar{\varphi}^r-\bar{u}^r|^2\leq 8C r^{2-2d}\]
which is possible because of (\ref{1}).

\medskip

Take the decomposition $\bar{\varphi}^r-\bar{u}^r=h+g$, where $h$ is harmonic in $B_{r_0}(0)$ and $h=\bar{\varphi}^r-\bar{u}^r$
on $\partial B_{r_0}(0)$.
By the mean value property of harmonic functions, we have
\[\sup_{B_{5/8}(0)}|h|\leq Cr^{1-d}.\]
Since $g=0$ on $\partial B_{r_0}(0)$ and
\[|\Delta g|\leq Cr^{2-d}=-\Delta(\frac{Cr^{2-d}}{4}(r_0^2-|z|^2),\]
comparison principle implies
\[\sup_{B_{5/8}(0)}|g|\leq Cr^{2-d}.\]
Combining these two we obtain
\[\sup_{B_{5/8}(0)}|\bar{\varphi}^r-\bar{u}^r|\leq Cr^{2-d}.\]
Combining with \eqref{rescaled equation}, by elliptic estimates we see
\begin{equation}\label{2}
\sup_{B_{9/16}(0)}|\nabla(\bar{\varphi}^r-\bar{u}^r)|\leq Cr^{2-d}.
\end{equation}
 By
\eqref{1},
\[|\{|\bar{\varphi}^r-\bar{u}^r|>r^{3/2-d}\}\cap B_{9/16}(0)|\leq Cr^{-1}.\]
In particular, for any ball $B_{Mr^{-1/2}}(x)\subset B_{9/16}(0)$
where $M$ is a large constant, there exists $y\in
B_{Mr^{-1/2}}(x)\cap \{|\bar{\varphi}^r-\bar{u}^r|<r^{3/2-d}\}$.
Integrating along the segment from $y$ to $x$ and using \eqref{2}, we get
\[|\bar{\varphi}^r(x)-\bar{u}^r(x)|\leq Cr^{3/2-d}~~\text{for any}~~~x\in B_{1/2}(0).\]
Rescaling back we can finish the proof.
\end{proof}

\medskip

\begin{lem}
There exists a constant $C$, independent of $r$ and $R$, such that for any $r\in(0,R/2)$,
\begin{equation}\label{3}
\sup_{B_{r/2}(0)}|\nabla^2\varphi^r-\nabla^2\varphi^{2r}|\leq
\frac{C}{r}.
\end{equation}
\end{lem}

\medskip

\begin{proof}
By \eqref{comparison of energy} we get
\[\int_{B_r(0)}|\nabla\varphi^r-\nabla\varphi^{2r}|^2\leq Cr^2.\]
Since both $\varphi^r$ and $\varphi^{2r}$ are harmonic, by interior gradient estimates we obtain the claim.
\end{proof}
\begin{lem}
For any $r\in(0,R)$,
\[\sup_{B_{r/2}(0)}|\varphi-u^R|\leq Cr^{3/2}.\]
\end{lem}
\begin{proof}
Take an $i_0$ such that $R/2<2^{i_0}r\leq R$. Checking the proof of
the previous lemma we see
\[\sup_{B_{2^{i_0-1}r}(0)}|\nabla^2\varphi^{2^{i_0}r}-\nabla^2\varphi|\leq
\frac{C}{2^{i_0}r}.\] Adding this and \eqref{3} from $i=1$ to $i=i_0$
we get
\begin{equation}\label{5}
\sup_{B_{r/2}(0)}|\nabla^2\varphi^r-\nabla^2\varphi|\leq \frac{C}{r}.
\end{equation}
Since for each $r$, $\varphi^r$ has the same symmetry as
$\varphi$ and it is harmonic (recall that the degree of $\varphi$, $d\geq 2$), we have
\[\varphi^r(0)=\varphi(0)=0,~~~\nabla\varphi^r(0)=\nabla\varphi(0)=0.\]
Integrating \eqref{5} twice we obtain,
\begin{equation}\label{4}
\sup_{B_{r/2}(0)}|\varphi^r-\varphi|\leq Cr.
\end{equation}
This combined with Lemma \ref{lem 2} implies the required claim.
\end{proof}

\medskip

A direct corollary of this lemma is the uniform boundedness of $u^R$
on any compact set. Hence we can take the limit
$u_\infty:=\lim\limits_{R\to+\infty}u^R$ which  is a solution of \eqref{equation} on the entire $\mathbb{R}^2$, enjoying the same symmetry
as $\varphi$,  $ \{ u_\infty >0 \}= \{ \varphi >0\}$, and satisfies
\[|u_\infty (x,y)-\varphi(x,y)|\leq C(|x|+|y|)^{3/2}.\]
In particular, $u_\infty$ is unbounded and grows at least quadratically.

This proves Theorem \ref{thm2}.

\begin{rmk}
By \cite{A-C-M}, there exists a second solution $u$ of \eqref{equation} satisfying the symmetry $u(Gz)=-u(z)$, which is bounded in $\mathbb{R}^2$. For example, if $F(u)=1+\cos u$,
we can construct a solution such that $-\pi<u<\pi$ in $\mathbb{R}^2$. In fact, in this case, if we modify $F(u)$ outside $[-\pi,\pi]$ to get a standard double-well potential, it becomes exactly the problem studied in \cite{A-C-M, D-F-P}. The bounded solution produced by this method takes values in $(-\pi,\pi)$ and it is still the solution of the original problem \eqref{equation}.
\end{rmk}


\section{proof of the improvement estimate (\ref{newest})}
\setcounter{equation}{0}

Let $u$ be the solution constructed in the previous section. Written in the exponential polar coordinate
 $(r,\theta)=(e^t,\theta)$, $u$ satisfies
\[\partial_t^2u+\partial_{\theta}^2u+e^{2t}f(u)=0.\]

\medskip

Let $v(t,\theta)=e^{-dt}u(t,\theta)$. Then $ v(t, \theta)$ satisfies
\begin{equation}\label{polar equation}
\partial_t^2v+2d\partial_tv+d^2v+\partial_\theta^2v+e^{(2-d)t}f(e^{dt}v)=0.
\end{equation}

\medskip

By the error bound established in the previous section, for $t\geq0$,
\begin{equation}\label{decay 1}
|v(t,\theta)-\cos(d\theta)|\leq Ce^{(3/2-d)t}.
\end{equation}

\medskip

By interior gradient estimates, for any $\varepsilon>0$ there
exists a constant $C$ such that for any ball $B_1(t,\theta)\subset\mathbb{R}\times\mathbb{S}^1$ (with respect to the product metric on $\mathbb{R}
\times\mathbb{S}^1$) and $u\in C^2(B_1(t,\theta))$,
\begin{eqnarray}\label{gradient estimate}
\sup_{B_{1/2}(t,\theta)}|\partial_\theta u|+|\partial_tu|\leq \varepsilon\sup_{B_1(t,\theta)}|\partial_t^2u+\partial_\theta^2u|+\frac{C}{\varepsilon}\sup_{B_1(t,\theta)}|u|.
\end{eqnarray}
 Since $|e^{(2-d)t}f(e^{dt}v)|\leq Ce^{(2-d)t}$, applying \eqref{gradient estimate} to $v-\cos(d\theta)$
 with $\varepsilon=e^{-\frac{t}{4}}$ we get a constant $C$ such that
for all $t\geq 0$,
\begin{equation}\label{decay 2}
|\partial_\theta(v(t,\theta)-\cos(d\theta))|+|\partial_tv(t,\theta)|\leq Ce^{(\frac{7}{4}-d)t}.
\end{equation}
Differentiating \eqref{polar equation} in $t$ we get
\[\partial_t^2\partial_tv+2d\partial_t\partial_tv+d^2\partial_tv+\partial_\theta^2\partial_tv+e^{2t}f^{\prime}(e^{dt}v)\partial_tv=0.\]
By the bound on $\partial_t v$, we have $|e^{2t}f^{\prime}(e^{dt}v)\partial_tv|\leq Ce^{(\frac{7}{4}+2-d)t}$. By taking $\varepsilon=e^{-t}$ in \eqref{gradient estimate} we obtain
\[|\partial_\theta\partial_tv(t,\theta)|+|\partial_t^2v(t,\theta)|\leq Ce^{(\frac{7}{4}+1-d)t}.\]

\medskip

Substituting this and \eqref{decay 1}, \eqref{decay 2} into \eqref{polar equation}, we get
\begin{equation}\label{decay 3}
|\partial_\theta^2(v(t,\theta)-\cos(d\theta))|\leq Ce^{(\frac{7}{4}+1-d)t}.
\end{equation}
If $d\geq 3$, this gives the exponential convergence of $v$ to $\cos(d\theta)$ in $C^2(\mathbb{S}^1)$.

\medskip

Below we assume that $d\geq 3$.

\medskip

Let $v(t,\theta)=\sum_{j\geq0}c_j(t)\cos(j\theta)$ be the Fourier decomposition of $v(t,\cdot)$. Note that because $v$ is even in $\theta$, there are only terms $\cos(j\theta)$ appearing in this decomposition. Moreover, by our construction,
\[\sum_{j\geq 0}c_j(t)\cos(j\theta+\frac{j\pi}{d})=v(t,\theta+\frac{\pi}{d})=-v(t,\theta)=-\sum_{j\geq0}
c_j(t)\cos(j\theta),\] so
$c_j(t)=0$ if there is no nonnegative integer $k$ such that $j=(2k+1)d$. In particular,
\[c_j(t)\equiv 0~~\text{for}~~j<d.\]

\medskip

Hence below we concentrate on those $c_j(t)$ with $j=d$ and $j\geq 3d$.

\medskip

Multiplying \eqref{polar equation} by $\cos(j\theta)$ and integrating, we get the equation for $c_j(t)$
\[\partial_t^2c_j+2d\partial_tc_j+(d^2-j^2)c_j+e^{(2-d)t}\left(\int_0^{2\pi}f(e^{dt}v)\cos(j\theta)d\theta\right)=0.\]
Denote $g_j(t)=e^{(2-d)t}\left(
\int_0^{2\pi}f(e^{dt}v)\cos(j\theta)d\theta
\right)$. Since $\cos(d\theta)
$ has only non-degenerate critical points and $v(t,\theta)\to\cos(d\theta)$ in $C^2(\mathbb{S}^1)$ as $t\to+\infty$ (cf. \eqref{decay 2} and \eqref{decay 3}), for $t$ large, $v(t,\cdot)$ has only non-degenerate critical points. By the oscillatory integral estimate ( [Section 8.1, \cite{stein}]) we get a constant $C_j$ such that
\begin{equation}\label{4.1}
\int_0^{2\pi}f(e^{dt}v)\cos(j\theta)d\theta = O( e^{-\frac{dt}{2}}), \ \ \ |g_j(t)|\leq C_je^{(2-\frac{3d}{2})t}.
\end{equation}

\medskip

For $t\geq 0$, we have the representation formula
\begin{equation}\label{representation}
c_j(t)=A_je^{-(d+j)t}+B_je^{-(d-j)t}+e^{-(d-j)t}\int_t^{+\infty}e^{(d-j)s}\int_0^se^{(d+j)(\tau-s)}g_j(\tau)d\tau ds.
\end{equation}
Substituting \eqref{4.1} into this and integrating directly, we see the last integral is bounded by $\frac{C_j}{j^2}e^{(2-\frac{3d}{2})t}$. In particular, for $j=d$,
\begin{equation}\label{decay 4}
|c_d(t)-B_d|\leq Ce^{(2-\frac{3d}{2})t}.
\end{equation}
Here, by \eqref{decay 1}, $B_d=1$.

\medskip

 It remains to estimate $v^\|:=v-c_d(t)\cos(d\theta)$. First note that for $j>d$, $|c_j(t)|\leq Ce^{(3/2-d)t}$ by \eqref{decay 1}. Hence we must have $B_j=0$. Next we have
\begin{lem}
For $t$ large, when measured in $L^\infty(\mathbb{S}^1)$,
\[v^\flat:=\sum_{j>d}e^{-(d-j)t}\int_t^{+\infty}e^{(d-j)s}\int_0^se^{(d+j)(\tau-s)}g_j(\tau)\cos(j\theta
)d\tau ds=O(e^{(2-\frac{3d}{2})t}).\]
\end{lem}
\begin{proof}
Direct calculations give, for $t+\tau<2s$,
\[\sum_{j>d}e^{-(d-j)t}e^{(d-j)s}e^{(d+j)(\tau-s)}\cos(j\theta)=e^{t-(2d+2)s+(2d+1)\tau}\frac{\cos(d+1)\theta-e^{t+\tau-2s}\cos d\theta}{1-2e^{t+\tau-2s}\cos\theta+e^{2(t+\tau-2s)}}.\]
Using this kernel, $v^\flat$ can be written as
\[v^\flat(t,\theta)=\int_t^{+\infty}\int_0^s\int_0^{2\pi}e^{t-(2d+2)s+(2d+1)\tau}\frac{\cos(d+1)\theta-e^{t+\tau-2s}\cos d\theta}{1-2e^{t+\tau-2s}\cos\theta+e^{2(t+\tau-2s)}}g(\tau,\theta)d\theta d\tau ds,\]
where $g(\tau,\theta)=e^{(2-d)\tau}f(e^{d\tau}v(\tau,\theta))$.

Note that
\[\frac{\cos(d+1)\theta-e^{t+\tau-2s}\cos d\theta}{1-2e^{t+\tau-2s}\cos\theta+e^{2(t+\tau-2s)}}\]
is uniformly bounded in $C^3(\mathbb{S}^1)$ when $t+\tau-2s\leq 0$. Hence by the oscillatory integral estimate (\cite{stein}),
\[\int_0^{2\pi}\frac{\cos(d+1)\theta-e^{t+\tau-2s}\cos d\theta}{1-2e^{t+\tau-2s}\cos\theta+e^{2(t+\tau-2s)}}g(\tau,\theta)d\theta=O(e^{(2-\frac{3d}{2})\tau}).\]
Substituting this into the above representation formula of $v^\flat$ we finish the proof.
\end{proof}
\begin{lem}
For $t$ large, when measured in $L^\infty(\mathbb{S}^1)$,
\[v^\sharp:=v^\|-v^\flat=O(e^{-\frac{d-1+\sqrt{17d^2-2d+1}}{2}t}).\]
\end{lem}
\begin{proof}
Direct calculations show that
\[(\partial_t^2+2d\partial_t+\partial_\theta^2+d^2)v^\flat+e^{(2-d)t}f(e^{dt}v(t,\theta))=0.\]
Hence
\[(\partial_t^2+2d\partial_t+\partial_\theta^2+d^2)v^\sharp=0.\]
Multiplying by $v^\sharp$ and integrating on $\mathbb{S}^1$, we get
\[\frac{d^2}{dt^2}\int_0^{2\pi}(v^\sharp)^2 d \theta+(2d-2)\frac{d}{dt}\int_0^{2\pi}(v^\sharp)^2 d \theta +2d^2\int_0^{2\pi}(v^\sharp)^2 d \theta
-2\int_0^{2\pi}(\partial_\theta v^\sharp)^2 d \theta=0.\]
By our construction, for any $t\geq0$, $v^\sharp$ is orthogonal to $\cos(j\theta), \sin(j\theta)$ for every $|j|\leq 3d-1$.
Hence
\[\int_0^{2\pi}(\partial_\theta v^\sharp)^2 d \theta \geq(3d)^2  \int_0^{2\pi}(v^\sharp)^2 d \theta,\]
and
\[L\int_0^{2\pi}(v^\sharp)^2 d \theta :=\frac{d^2}{dt^2}\int_0^{2\pi}(v^\sharp)^2 d \theta +(2d-2)\frac{d}{dt}\int_0^{2\pi}(v^\sharp)^2 d \theta -16d^2\int_0^{2\pi}(v^\sharp)^2 d \theta
\geq0.\]
By \eqref{decay 1}, \eqref{decay 4} and the previous lemma, we have the decay estimate
\[\int_0^{2\pi}v^\sharp(t,\theta)^2d\theta\leq Ce^{(3-2d)t}.\]
Let $\left(\int_0^{2\pi}(v^\sharp(0,\theta)^2d\theta\right)e^{-(d-1+\sqrt{17d^2-2d+1})t}$ be a solution
of $Lh=0$ which has the same boundary value at $t=0$ and $t=+\infty$. By the comparison principle we get
 for any $t\geq0$,
\[\int_0^{2\pi}v^\sharp(t,\theta)^2d\theta\leq(\int_0^{2\pi}v^\sharp(0)^2 d \theta )e^{-(d-1+\sqrt{17d^2-2d+1})t}.
\]
Then by applying standard elliptic estimates to $v^\sharp$ we get its $L^\infty(\mathbb{S}^1)$ bound.
\end{proof}

For $d\geq 2$,
\[\frac{d-1+\sqrt{17d^2-2d+1}}{2}\geq\frac{3d}{2}-2.\]
Putting the above estimates together we see for every $d\geq 3$,
\[\sup|v(t,\theta)-\cos(d\theta)|\leq Ce^{-(\frac{3d}{2}-2)t}.\]
Coming back to $u$, we get a constant $C$ such that for all $z\in\mathbb{C}$
\[|u(z)-\varphi(z)|\leq C(1+|z|)^{-(\frac{d}{2}-2)}\]
which proves (\ref{newest}).

\end{document}